\documentclass[12pt,twoside]{amsart}
\usepackage{amssymb,amsmath,amscd,enumerate,verbatim,xcolor,mathtools,fullpage}

\usepackage[top=3.5cm, bottom=2.5cm, left=3.0cm, right=3cm]{geometry}
\usepackage{color, fancyhdr, graphicx, wrapfig}
\usepackage{multirow}
\usepackage{diagbox}
\usepackage[unicode]{hyperref}
\usepackage{indentfirst}
\usepackage{tikz}
\usepackage{array}
\usetikzlibrary{calc}
\usepackage{mathrsfs}
\usepackage{graphicx}
\usepackage{longtable}
\usepackage{nicematrix}
\usepackage{cleveref}
\usetikzlibrary{positioning}
\usepackage[sort&compress]{natbib}
\usepackage[all,cmtip]{xy}

\newcommand{\kk}{\Bbbk}

\numberwithin{equation}{section}

\newtheorem{thm}{Theorem}[section]
\newtheorem{lem}[thm]{Lemma}

\newtheorem{prop}[thm]{Proposition}

\theoremstyle{definition}
\newtheorem{defn}[thm]{Definition}

\newtheorem{ex}[thm]{Example}

\newtheorem{rem}[thm]{Remark}
\newtheorem{quest}[thm]{Question}

\newcommand{\abs}[1]{\left| #1\right|}

\usepackage{hyperref}
\hypersetup{
	colorlinks=true,       
	linkcolor=red,          
	citecolor=blue,        
	filecolor=magenta,      
	urlcolor=cyan           
}

\begin{document}
\date{}

\title{The Weak Lefschetz Property for Tensor Products of Artinian Monomial Algebras and Its Applications to Lollipop Graphs}

\author[T.Q. Hoa]{Tran Quang Hoa}
\address{University of Education, Hue University,  34 Le Loi St., Hue City, Vietnam.}
\email{tranquanghoa@hueuni.edu.vn}

\author[N.D. Phuoc]{Nguyen Duy Phuoc}
\address{University of Education, Hue University, 34 Le Loi St., Hue City, Viet Nam}
\email{ndphuoc@dhsphue.edu.vn}

\author[T.N.T. Son]{Tran Nguyen Thanh Son}
\address{University of Education, Hue University, 34 Le Loi St., Hue City, Viet Nam}

\email{tntson@dhsphue.edu.vn}

\begin{abstract}
In this paper, we investigate the weak Lefschetz property for tensor products of Artinian monomial algebras and complete quadratic monomial algebras. As an application, we classify the weak Lefschetz property of the Artinian algebras $A(L_{m,n})$, which are defined by the edge ideals of the lollipop graphs $L_{m,n}$ together with the squares of the variables.
\end{abstract}

\makeatletter
\@namedef{subjclassname@2020}{%
	\textup{2020} Mathematics Subject Classification}
\makeatother

\subjclass[2020]{13E10, 13F20, 13F55, 05C31}
\keywords{Artinian algebras, complete quadratic monomial algebras, edge ideals, independence polynomials, lollipop graphs, weak Lefschetz property}

\maketitle
\section{Introduction}
A graded Artinian algebra $A=\bigoplus_{i=0}^D [A]_i$ over a field $\kk$ is said to have the \emph{weak Lefschetz property} (WLP for short) if there exists a linear form $\ell \in [A]_1$ such that each multiplication map 
\[
\cdot \ell:[A]_i \longrightarrow [A]_{i+1}
\]
has maximal rank for all $i$. Similarly, $A$ has the \emph{strong Lefschetz property} (SLP for short) if there exists a linear form $\ell$ such that each multiplication map
\[
\cdot \ell^j:[A]_i \longrightarrow [A]_{i+j}
\]
has maximal rank for all $i$ and all $j$.  
The Lefschetz properties have been an active topic of research in commutative algebra in recent years, owing to their connections with several other areas of mathematics, such as the Dilworth number in poset theory and Schur–Weyl duality; see \cite{HMMNWW2013} and \cite{MN2013} for an overview.

For Artinian $\kk$-algebras defined by monomial ideals, the literature concerning their Lefschetz properties is quite extensive (see, for instance, \cite{AB2020,AN2018,DaoNair2022,GLN2020,Holleben2024,MM2016,MMO2013,MMN2011,MNS2020,HT2023} and the references therein).  
In this work, we focus on a special class of Artinian algebras defined by quadratic monomials, described as follows.  
Let $G=(V,E)$ be a simple graph on the vertex set $V=\{1,2,\ldots,n\}$ and $R=\kk[x_1,\ldots,x_n]$ be the standard graded polynomial ring over $\kk$.  
The \emph{edge ideal} of $G$ is given by
\[
I(G)=(x_i x_j \mid \{i,j\}\in E) \subset R.
\]
We then define
\[
A(G)=\frac{R}{(x_1^2,\ldots,x_n^2)+I(G)},
\]
which we call the \emph{Artinian algebra associated to $G$}.  
We are interested in the following question.

\begin{quest}\label{quest_WLPofAG}
	For which graphs $G$, does $A(G)$ have the WLP or the SLP? If $A(G)$ does not have the WLP or the SLP, in which degrees do the multiplication maps fail to have maximal rank?
\end{quest}

The algebra $A(G)$ has been studied in \cite{NT2024,QHT2020}, where the WLP and SLP were classified for several special classes of graphs, including complete graphs, star graphs, Barbell graphs, wheel graphs, paths, and cycles.  
Artinian algebras defined by quadratic monomials were also considered in earlier works by Michałek–Miró-Roig \cite{MM2016}, Migliore–Nagel–Schenck \cite{MNS2020}, Dao–Nair \cite{DaoNair2022}, Jonsson Kling \cite{Kling2024}, Holleben \cite{Holleben2024,Holleben2025}, and Cooper et al. \cite{Cooperetal2023}.

Our main objective in this paper is to give a complete classification of the lollipop graphs $L_{m,n}$ for which the associated Artinian algebras have (or fail to have) the WLP.  
As a first step, we consider a broader framework—the tensor product of an Artinian monomial algebra with a complete quadratic monomial algebra.  
This leads to our first main result, stated below.

\begin{thm}[Theorem \ref{generalization of lollipop}] \label{Theorem1.2}
	Suppose that $\operatorname{char}(\Bbbk)=0$. Let 
	\[
	B:= \frac{\Bbbk[x_1,x_2,\ldots,x_n]}{(x_1,x_2,\ldots,x_n)^2} \otimes_\Bbbk A,
	\]
	where $A=\dfrac{\Bbbk[y_1,y_2,\ldots,y_m]}{J}$ is an Artinian monomial algebra over $\Bbbk$ of socle degree $D>0$.  
	Set $\ell=y_1+y_2+\cdots+y_m$ and $\ell'= x_1+x_2+\cdots + x_n+\ell$. Then, for all $i\in \{1,2,\ldots,D-1\}$, the multiplication map 
	\[
	\cdot \ell':[B]_i \longrightarrow [B]_{i+1}
	\]
	is injective (respectively, surjective) if and only if the multiplication maps
	\[
	\cdot \ell:[A]_{i-1}\longrightarrow[A]_i \quad \text{and}\quad \cdot \ell^2:[A]_{i-1}\longrightarrow[A]_{i+1}
	\]
	are both injective (respectively, surjective).  
	In particular, the multiplication map $\cdot \ell':[B]_0\longrightarrow[B]_1$ is always injective, while the map $\cdot\ell':[B]_D\longrightarrow[B]_{D+1}$ has maximal rank if and only if the map $\cdot\ell:[A]_{D-1}\longrightarrow[A]_D$ is surjective.
\end{thm}

This theorem also provides a partial answer to a question posed by Jonsson Kling \cite{K2025}, which concerns the Artinian $\Bbbk$-algebra 
\[
\frac{\Bbbk[x_1,x_2,\ldots,x_n]}{(x_1,x_2,\ldots,x_n)^d}\otimes_\Bbbk A.
\]

As a consequence, we completely determine all Artinian $\Bbbk$-algebras $A(L_{m,n})$ associated to lollipop graphs $L_{m,n}$ that fail to have the WLP, by applying Theorem~\ref{Theorem1.2} together with suitable short exact sequences.  
For the remaining cases, we show that $A(L_{m,n})$ has the WLP by using the snake lemma in combination with inductive arguments.  
Our second main result is stated below.

\begin{thm} [Theorem \ref{WLP for lollipop}]\label{SecondTheorem}
	Suppose that $\operatorname{char}(\Bbbk)=0$. Then $A(L_{m,n})$ has the WLP if and only if one of the following holds:
	\begin{enumerate}[\quad \rm (i)]
		\item $m=1$ and $n\in \{1,2,\dots,6,8,9,12\}$;
		\item $m=2$ and $n\in \{1,2,\dots,5,7,8,11\}$;
		\item $m\geq 3$ and $n\in \{1,3,4,7\}$.
	\end{enumerate}
\end{thm}

The paper is organized as follows.  
Section~2 reviews the necessary terminology and background results concerning Artinian algebras, the weak Lefschetz property, and basic graph theory.  
Section~3 is devoted to studying the WLP for tensor products of Artinian monomial algebras with complete quadratic monomial algebras (see Theorem~\ref{generalization of lollipop}).  
In Section~4, we investigate the unimodality of the independence polynomials of lollipop graphs and determine their modes, which play a key role in analyzing the WLP of the corresponding Artinian algebras.  
Finally, Section~5 is devoted to the proof of \Cref{SecondTheorem}.

\section{Preliminaries}
This section reviews standard terminology and notation from commutative and combinatorial commutative algebra, along with several auxiliary results that will be used in the subsequent sections.

\subsection{The weak Lefschetz property}

Let $R=\kk[x_1,\ldots,x_n]$ be the standard graded polynomial ring over a field $\kk$, where each $x_i$ has degree $1$, and let $I \subset R$ be a homogeneous ideal such that $A=R/I$ is Artinian. Then $A$ is a graded Artinian algebra, which can be decomposed as
\[
A=\bigoplus_{i=0}^{D}[A]_i.
\]

\begin{defn}
	We say that $A$ has the \emph{weak Lefschetz property} (WLP) if there exists a linear form $\ell \in [A]_1$ such that, for every integer $j$, the multiplication map
	\[
	\cdot \ell: [A]_j \longrightarrow [A]_{j+1}
	\]
	has maximal rank; that is, it is either injective or surjective. In this case, such an $\ell$ is called a \emph{Lefschetz element} of $A$. 
\end{defn}

Throughout this paper, we restrict our attention to Artinian algebras defined by monomial ideals. In this setting, it suffices to choose the Lefschetz element as the sum of all variables.

\begin{prop}[{\rm \cite[Proposition~2.2]{MMN2011}}]\label{Proposition2.5}
	Let $I \subset R=\Bbbk[x_1,\ldots,x_n]$ be an Artinian monomial ideal. Then the algebra $A=R/I$ has the WLP if and only if $\ell=x_1+x_2+\cdots+x_n$ is a Lefschetz element of $A$.
\end{prop}

Next, we present a necessary condition for an Artinian graded algebra to have the WLP.  
Before that, we introduce some notation.
\begin{defn}
	Let $\Bbbk$ be a field and $A= \bigoplus_{j\geq 0} [A]_j$ be a standard graded $\Bbbk$-algebra. The \emph{Hilbert series} of $A$ is the power series 
	\[
	HS(A,t) = \sum_{j\geq 0} \dim_\Bbbk [A]_j \, t^j.
	\]
	If $A$ is Artinian, then $[A]_i=0$ for $i \gg 0$. We denote 
	\[
	D := \max \{i \mid [A]_i \neq 0\},
	\]
	and we call $D$ the \emph{socle degree} of $A$. 
	In this case, the Hilbert series of $A$ is a polynomial
	\[
	HS(A,t) = 1 + h_1 t + \cdots + h_D t^D,
	\]
	where $h_i = \dim_\Bbbk [A]_i > 0$. 
\end{defn}

\begin{defn}\label{mode}
	A polynomial $\sum_{k=0}^n a_k t^k$ with non-negative coefficients is called \emph{unimodal} if there exists an integer $m$ such that 
	\[
	a_0 \leq a_1 \leq \cdots \leq a_{m-1} \leq a_m \geq a_{m+1} \geq \cdots \geq a_n.
	\]
	Set $a_{-1}=0$. The \emph{mode} of a unimodal polynomial $\sum_{k=0}^n a_k t^k$ is defined to be the unique integer $i \in \{0,\dots,n\}$ such that 
	\[
	a_{i-1} < a_i \geq a_{i+1} \geq \cdots \geq a_n.
	\]
\end{defn}

\begin{prop}[{\rm \cite[Proposition~3.2]{HMMNWW2013}}]\label{Proposition2.8}
	If $A$ has the WLP, then the Hilbert series of $A$ is unimodal.
\end{prop}

Finally, to study the failure of the WLP of tensor products of $\kk$-algebras, the following simple lemma turns out to be quite useful. 

\begin{lem}{\rm\cite[Lemma 7.8]{BMMNZ12}}\label{lem_tensor}
	Let $A=A'\otimes_\kk A''$ be a tensor product of two graded Artinian $\kk$-algebras $A'$ and $A''$. Let $\ell'\in A'$ and $\ell''\in A''$ be linear elements, and set $\ell=\ell'+\ell''=\ell'\otimes 1 + 1\otimes \ell''\in A$. Then
	\begin{enumerate}
		\item [\rm (a)] If the multiplication maps $\cdot\ell': [A']_{i}\longrightarrow [A']_{i+1}$ and $\cdot\ell'': [A'']_{j}\longrightarrow [A'']_{j+1}$ are both not surjective, then neither is the map
		$$
		\cdot\ell: [A]_{i+j+1}\longrightarrow [A]_{i+j+2}.
		$$
		\item [\rm (b)] If the multiplication maps $\cdot\ell': [A']_{i}\longrightarrow [A']_{i+1}$ and $\cdot\ell'': [A'']_{j}\longrightarrow [A'']_{j+1}$ are both not injective, then neither is the map
		$$
		\cdot\ell: [A]_{i+j}\longrightarrow [A]_{i+j+1}.
		$$
		
	\end{enumerate}
\end{lem}


\subsection{Graph theory}

Throughout this subsection, by a \emph{graph} we mean a simple, undirected graph $G=(V,E)$ with vertex set $V=V(G)$ and edge set $E=E(G)$. 
We begin by recalling some basic definitions.
\begin{defn}
The \emph{disjoint union} of two graphs $G_1$ and $G_2$ is the graph $G = G_1 \cup G_2$ whose vertex set is the disjoint union of $V(G_1)$ and $V(G_2)$, and whose edge set is the disjoint union of $E(G_1)$ and $E(G_2)$. 
\end{defn}
\begin{defn}
Let $G=(V,E)$ be a graph with vertex set $V=\{1,2,\dots,n\}$.
\begin{itemize}
\item[(a)] A subset $X \subset V$ is called an \emph{independent set} if, for all $i,j \in X$, we have $\{i,j\} \notin E$. An independent set $X$ of cardinality $k$ is called a \emph{$k$-independent set} of $G$.	
\item[(b)] The \emph{independence number} of $G$ is the maximum cardinality of an independent set in $G$, denoted by $\alpha(G)$. 
\item[(c)] The \emph{independence polynomial} of $G$ is the polynomial in the variable $t$ whose coefficient of $t^k$ is the number of $k$-independent sets. We denote it by $I(G;t)$, that is,
\[
I(G;t)= \sum_{k=0}^{\alpha(G)} s_k(G)t^k,
\]
where $s_k(G)$ denotes the number of $k$-independent sets of $G$.
\end{itemize}
\end{defn}

The independence polynomial of a graph was defined by Gutman and Harary in \cite{GH83} as a generalization of the matching polynomial of a graph.
For a vertex $v \in V$, the \emph{open neighborhood} of $v$ is defined by
\[
N(v) = \{\, u \in V \mid \{u,v\} \in E \,\},
\]
and the \emph{closed neighborhood} of $v$ is given by
\[
N[v] = N(v) \cup \{v\}.
\]

The following result is useful for computing independence polynomials of graphs.

\begin{prop}[{\rm \cite[Theorem~2.3 and Corollary~3.3]{HL94}}]\label{FormulaforIndependence}
Let $G,G_1,G_2$ be graphs and $v$ be a vertex of $G$. Then the following equalities hold:
\begin{itemize}
\item[(a)] $I(G;t) = I(G \setminus v;t) + t \, I(G \setminus N[v];t)$;
\item[(b)] $I(G_1 \cup G_2;t) = I(G_1;t) \cdot I(G_2;t)$.  
\end{itemize}
\end{prop}

\subsection{Artinian algebras associated to graphs} 

Let $G=(V,E)$ be a graph with vertex set $V=\{1,2,\dots,n\}$. 
Let $R=\Bbbk[x_1,x_2,\dots,x_n]$ be the standard graded polynomial ring over a field $\Bbbk$. 
The \emph{edge ideal} of $G$ is defined by
\[
I(G)=(x_ix_j\mid \{i,j\}\in E)\subset R.
\]
Then, we say that 
$$A(G)=\frac{R}{(x_1^2,\ldots,x_n^2)+I(G)}$$
is the {\it artinian algebra associated to $G$}. The algebra $A(G)$ contains significant combinatorial information about $G$, as witnessed by
\begin{prop}[{\rm \cite[Proposition~2.10]{NT2024}}]
	The Hilbert series of $A(G)$ coincides with the independence polynomial of $G$, that is,
	\[
	HS(A(G),t) = I(G;t).
	\]
\end{prop}

Nguyen and Tran~\cite{NT2024} studied the independence polynomials of path graphs $P_n$ on $n$ vertices as well as their associated Artinian algebras. 
We recall some necessary results later. Denote by $\lambda_n$ the mode of the independence polynomial of $P_n$. 
\begin{table}[!h]
	\caption{Mode of $I(P_n;t)$}
	\label{tab_indpoly}
	\begin{tabular}{| c  | c  | c | c | c | c | c | c | c |  c | c |  c  |  c  | c  | c | c | c| c| c| c| c| c| c| }
		\hline
		$n$ &  1 & 2 & 3 & 4 & 5 & 6 & 7 & 8 & 9 & 10 & 11 & 12 & 13 &14&15 & 16 & 17 & 18 & 19 &20 \\
		\hline
		$\lambda_n$        &     0  & 1  & 1  & 1  & 2  & 2 & 2 & 2 & 3  & 3 & 3 & 4 & 4 & 4 & 4& 5& 5 & 5 & 5 &6\\
		\hline
	\end{tabular}
\end{table}
Table \ref{tab_indpoly} provides information about the initial values of the mode $\lambda_n$ of  the independence polynomial $I(P_n;t)$, by using {\tt Macaulay2} \cite{M2codes}.

\begin{lem}[{\rm \cite[Lemma~3.2]{NT2024}}]\label{lambda_n < lambda_n+1}
For any $n\geq 1$, $\lambda_n\le \lambda_{n+1}\le \lambda_n+1.$
\end{lem}

\begin{thm}[{\rm \cite[Theorem~4.2 and Proposition~4.3]{NT2024}}]\label{WLP for P_n}
	Suppose that $\operatorname{char}(\Bbbk) = 0$. Then the algebra $A(P_n)$ has the WLP if and only if $n \in \{1, 2, \dots, 7, 9, 10, 13\}.$
	Furthermore:
	\begin{enumerate}[\quad \rm (i)]
		\item For any $n \ge 17$, $A(P_n)$ fails surjectivity from degree $\lambda_n$ to degree $\lambda_n + 1$.
		\item If $n \ge 12$ and $\lambda_n = \lambda_{n-1} + 1$, then $A(P_n)$ fails injectivity from degree $\lambda_n - 1$ to degree $\lambda_n$.
		\item If $n \in \{8, 11, 14, 15, 17\}$, then $A(P_n)$ fails only surjectivity from degree $\lambda_n$ to degree $\lambda_n + 1$.
	\end{enumerate}
\end{thm}


\section{WLP of tensor products of Artinian monomial algebras and complete quadratic monomial algebras}
Let 
\[
A = \bigoplus_{i=0}^D [A]_i = \frac{\Bbbk[y_1, y_2, \ldots, y_m]}{J}
\]
be an Artinian monomial $\Bbbk$-algebra of socle degree $D > 0$.  
Denote by $h_i = \dim_\Bbbk [A]_i$ the $i$-th entry of the Hilbert function of $A$, and let $\ell = y_1 + y_2 + \cdots + y_m$ be the sum of the variables.  
In this section, we study the weak Lefschetz property of the tensor product
\[
B := \frac{\Bbbk[x_1, x_2, \ldots, x_n]}{(x_1, x_2, \ldots, x_n)^2} \otimes_\Bbbk A.
\]

For each $i = 0, 1, \dots, D$, let $\mathcal{A}_i$ be a $\Bbbk$-basis of $[A]_i$, and let $\mathcal{M}^{\,i}_{i+t}$ denote the matrix representing the linear map
\[
\cdot \ell^{\,t} : [A]_i \longrightarrow [A]_{i+t}
\]
with respect to these bases.  
It is clear that the socle degree of $B$ is $D + 1$, so we may write $B = \bigoplus_{i=0}^{D+1} [B]_i$.  
For each $i = 1, 2, \dots, D$, a $\Bbbk$-basis of $[B]_i$ is given by
\[
\mathcal{B}_i := x_1 \mathcal{A}_{i-1} \sqcup x_2 \mathcal{A}_{i-1} \sqcup \cdots \sqcup x_n \mathcal{A}_{i-1} \sqcup \mathcal{A}_i.
\]
In particular, $\mathcal{B}_0 := \mathcal{A}_0$ is a basis of $[B]_0$, and $\mathcal{B}_{D+1} := \bigsqcup_{i=1}^n x_i \mathcal{A}_D$ is a basis of $[B]_{D+1}$.

\begin{lem}\label{block matrix}
	Set $\ell' = x_1 + x_2 + \cdots + x_n + \ell$.  
	For every $i = 1, 2, \dots, D - 1$, the matrix representing the linear map
	\[
	\cdot \ell' : [B]_i \longrightarrow [B]_{i+1}
	\]
	with respect to the pair of bases $(\mathcal{B}_i, \mathcal{B}_{i+1})$ has the following block form:
	\begin{equation}\label{3.1}
		\begin{bNiceMatrix}[margin]
			\Block[borders={bottom,right}]{4-4}{} 
			\mathcal{M}^{\,i-1}_i & 0_{h_i,h_{i-1}} & \cdots & 0_{h_i,h_{i-1}} & I_{h_i}\\
			0_{h_i,h_{i-1}} & \mathcal{M}^{\,i-1}_i & \cdots & 0_{h_i,h_{i-1}} & I_{h_i}\\
			\vdots & \vdots & \ddots & \vdots & \vdots\\
			0_{h_i,h_{i-1}} & 0_{h_i,h_{i-1}} & \cdots & \mathcal{M}^{\,i-1}_i & I_{h_i}\\
			0_{h_{i+1},h_{i-1}} & 0_{h_{i+1},h_{i-1}} & \cdots & 0_{h_{i+1},h_{i-1}} &
			\Block[borders={top,left}]{1-1}{} \mathcal{M}^{\,i}_{i+1}
		\end{bNiceMatrix},
	\end{equation}
	where $0_{m,n}$ denotes the zero matrix of size $m \times n$, and $I_h$ denotes the identity matrix of size $h \times h$.  
	
	In particular:
	\begin{itemize}
		\item The representation matrix of $\cdot \ell' : [B]_0 \longrightarrow [B]_1$ is
		\begin{equation}\label{3.2}
			\begin{bNiceMatrix}[margin]
				I_1\\
				I_1\\
				\vdots\\
				I_1\\
				\mathcal{M}^{\,0}_1
			\end{bNiceMatrix},
		\end{equation}
		which has size $(n + h_1) \times 1$.
		
		\item The representation matrix of $\cdot \ell' : [B]_D \longrightarrow [B]_{D+1}$ is
		\begin{equation}\label{3.3}
			\begin{bNiceMatrix}
				\Block[borders={right}]{4-4}{} 
				\mathcal{M}^{\,D-1}_D & 0_{h_D,h_{D-1}} & \cdots & 0_{h_D,h_{D-1}} & I_{h_D}\\
				0_{h_D,h_{D-1}} & \mathcal{M}^{\,D-1}_D & \cdots & 0_{h_D,h_{D-1}} & I_{h_D}\\
				\vdots & \vdots & \ddots & \vdots & \vdots\\
				0_{h_D,h_{D-1}} & 0_{h_D,h_{D-1}} & \cdots & \mathcal{M}^{\,D-1}_D & I_{h_D}
			\end{bNiceMatrix},
		\end{equation}
		which has size $n h_D \times (n h_{D-1} + h_D)$.
	\end{itemize}
\end{lem}

\begin{proof}
Fix $i \in \{1,2,\dots,D-1\}$.  For each $1 \le j \le n$, we have
\[
\ell' \cdot (x_j \mathcal{A}_{i-1}) = x_j (\ell \cdot \mathcal{A}_{i-1}),
\]
which contributes the block $\mathcal{M}^{i-1}_i$ in the $j$-th diagonal position of~\eqref{3.1}.  
	
On the other hand,
\[
	\ell' \cdot \mathcal{A}_i = x_1 \mathcal{A}_i + x_2 \mathcal{A}_i + \cdots + x_n \mathcal{A}_i + \ell \cdot \mathcal{A}_i,
	\]
	which yields the identity blocks $I_{h_i}$ in the last column of~\eqref{3.1}, together with the block $\mathcal{M}^i_{i+1}$ in the lower-right corner.  
	Hence, the matrix representing
	\[
	\cdot \ell' : [B]_i \longrightarrow [B]_{i+1}
	\]
	with respect to the pair of bases $(\mathcal{B}_i, \mathcal{B}_{i+1})$ has the block form~\eqref{3.1}.  
	
	Finally, since $\mathcal{B}_0 = \mathcal{A}_0$ and $\mathcal{B}_{D+1} = \bigsqcup_{j=1}^n x_j \mathcal{A}_D$, the same argument shows that the matrices representing
	\[
	\cdot \ell' : [B]_0 \longrightarrow [B]_1
	\quad\text{and}\quad
	\cdot \ell' : [B]_D \longrightarrow [B]_{D+1}
	\]
	are precisely of the block forms~\eqref{3.2} and~\eqref{3.3}, respectively.
\end{proof}

\begin{thm}\label{generalization of lollipop}
Assume that $\operatorname{char}(\Bbbk) = 0$.	For every $i \in \{1,2,\ldots,D-1\}$, the multiplication map
	\[
	\cdot \ell' : [B]_i \longrightarrow [B]_{i+1}
	\]
	is injective (respectively, surjective) if and only if the multiplication maps
	\[
	\cdot \ell : [A]_{i-1} \longrightarrow [A]_i 
	\quad\text{and}\quad 
	\cdot \ell^2 : [A]_{i-1} \longrightarrow [A]_{i+1}
	\]
	are both injective (respectively, surjective).
	
	In particular, the multiplication map $\cdot \ell' : [B]_0 \to [B]_1$ is always injective, whereas the map $\cdot \ell' : [B]_D \to [B]_{D+1}$ has maximal rank precisely when $\cdot \ell : [A]_{D-1} \to [A]_D$ is surjective.
\end{thm}

\begin{proof}
The case $i = 0$ follows directly from Lemma~\ref{block matrix}.  
Indeed, by~\eqref{3.2}, the matrix of
\[
\cdot \ell' : [B]_0 \longrightarrow [B]_1
\]
is a nonzero column vector, and hence the map is injective.  
	
Now fix $i \in \{1, 2, \ldots, D-1\}$. 
By Lemma~\ref{block matrix}, the matrix of the multiplication map
\[
\cdot \ell' : [B]_i \longrightarrow [B]_{i+1}
\]
has the block form~\eqref{3.1}. 
By performing row operations, the rank of matrix~\eqref{3.1} is equal to the rank of the following matrix
\begin{equation}
	\begin{bNiceMatrix}[margin]
		\Block[borders={bottom,right}]{4-4}{} 
		\mathcal{M}^{i-1}_i & 0_{h_i,h_{i-1}} & \cdots & 0_{h_i,h_{i-1}} & I_{h_i}\\
		-\mathcal{M}^{i-1}_i & \mathcal{M}^{i-1}_i & \cdots & 0_{h_i,h_{i-1}} & 0_{h_i,h_i}\\
		\vdots & \vdots & \ddots & \vdots & \vdots\\
		-\mathcal{M}^{i-1}_i & 0_{h_i,h_{i-1}} & \cdots & \mathcal{M}^{i-1}_i & 0_{h_i,h_i}\\
		0_{h_{i+1},h_{i-1}} & 0_{h_{i+1},h_{i-1}} & \cdots & 0_{h_{i+1},h_{i-1}} & 
		\Block[borders={top,left}]{1-1}{}\mathcal{M}^i_{i+1} \label{3.4}
	\end{bNiceMatrix}.
\end{equation}

By column operations, the rank of matrix~\eqref{3.4} is equal to the rank of the following matrix
\begin{equation}
	\begin{bNiceMatrix}[margin]
		\Block[borders={bottom,right}]{4-4}{} 
		\mathcal{M}^{i-1}_i & 0_{h_i,h_{i-1}} & \cdots & 0_{h_i,h_{i-1}} & I_{h_i}\\
		0_{h_i,h_{i-1}} & \mathcal{M}^{i-1}_i & \cdots & 0_{h_i,h_{i-1}} & 0_{h_i,h_i}\\
		\vdots & \vdots & \ddots & \vdots & \vdots\\
		0_{h_i,h_{i-1}} & 0_{h_i,h_{i-1}} & \cdots & \mathcal{M}^{i-1}_i & 0_{h_i,h_i}\\
		0_{h_{i+1},h_{i-1}} & 0_{h_{i+1},h_{i-1}} & \cdots & 0_{h_{i+1},h_{i-1}} & 
		\Block[borders={top,left}]{1-1}{}\mathcal{M}^i_{i+1} \label{3.5}
	\end{bNiceMatrix}.
\end{equation}

Applying further column operations, the rank of matrix~\eqref{3.5} is equal to the rank of the following matrix
\begin{equation}
	\begin{bNiceMatrix}[margin]
		\Block[borders={bottom,right}]{4-4}{} 
		0_{h_i,h_{i-1}} & 0_{h_i,h_{i-1}} & \cdots & 0_{h_i,h_{i-1}} & I_{h_i}\\
		0_{h_i,h_{i-1}} & \mathcal{M}^{i-1}_i & \cdots & 0_{h_i,h_{i-1}} & 0_{h_i,h_i}\\
		\vdots & \vdots & \ddots & \vdots & \vdots\\
		0_{h_i,h_{i-1}} & 0_{h_i,h_{i-1}} & \cdots & \mathcal{M}^{i-1}_i & 0_{h_i,h_i}\\
		-\mathcal{M}^i_{i+1}\mathcal{M}^{i-1}_i & 0_{h_{i+1},h_{i-1}} & \cdots & 0_{h_{i+1},h_{i-1}} & 
		\Block[borders={top,left}]{1-1}{}\mathcal{M}^i_{i+1} \label{3.6}
	\end{bNiceMatrix}.
\end{equation}

Finally, by row operations, the rank of matrix~\eqref{3.6} is equal to the rank of the following matrix
\begin{equation}
	\begin{bNiceMatrix}[margin]
		\Block[borders={bottom,right}]{4-4}{} 
		0_{h_i,h_{i-1}} & 0_{h_i,h_{i-1}} & \cdots & 0_{h_i,h_{i-1}} & I_{h_i}\\
		0_{h_i,h_{i-1}} & 0_{h_i,h_{i-1}} & \cdots & \mathcal{M}^{i-1}_i & 0_{h_i,h_i}\\
		\vdots & \vdots & \ddots & \vdots & \vdots\\
		0_{h_i,h_{i-1}} & \mathcal{M}^{i-1}_i & \cdots & 0_{h_i,h_i} & 0_{h_i,h_i}\\
		\mathcal{M}^i_{i+1}\mathcal{M}^{i-1}_i & 0_{h_{i+1},h_{i-1}} & \cdots & 0_{h_{i+1},h_{i-1}} & 
		\Block[borders={top,left}]{1-1}{}0_{h_{i+1},h_i} \label{3.7}
	\end{bNiceMatrix}.
\end{equation}

Note that the matrix $\mathcal{M}^i_{i+1}\mathcal{M}^{i-1}_i$ represents the multiplication map 
\[
\cdot \ell^2 : [A]_{i-1} \longrightarrow [A]_{i+1}.
\]
Hence, the map $\cdot \ell' : [B]_i \longrightarrow [B]_{i+1}$ is injective if and only if matrix~\eqref{3.7} has maximal column rank. 
Equivalently, both matrices $\mathcal{M}^{i-1}_i$ and $\mathcal{M}^i_{i+1}\mathcal{M}^{i-1}_i$ have maximal column rank. 
In other words, the maps 
\[
\cdot \ell : [A]_{i-1} \longrightarrow [A]_i 
\quad \text{and} \quad 
\cdot \ell^2 : [A]_{i-1} \longrightarrow [A]_{i+1}
\]
are injective. 
The case of surjectivity follows in a similar manner.

Finally, for $i = D$, the matrix of 
	\(\cdot \ell' : [B]_D \to [B]_{D+1}\) 
	has the form~\eqref{3.3}.  By performing similar row and column operations as above, we deduce that the rank of matrix~\eqref{3.3} is equal to the rank of the following matrix
	\begin{equation}
		\begin{bNiceMatrix}
			\Block[borders={right}]{4-4}{}0_{h_D,h_{D-1}} & 0_{h_D,h_{D-1}} & \cdots & 0_{h_D,h_{D-1}} & I_{h_D}\\
			0_{h_D,h_{D-1}} & 0_{h_D,h_{D-1}} & \cdots & \mathcal{M}^{D-1}_D & 0_{h_D,h_D}\\
			\vdots & \vdots & \ddots & \vdots & \vdots\\
			0_{h_D,h_{D-1}} & \mathcal{M}^{D-1}_D & \cdots & 0_{h_D,h_{D-1}} & 0_{h_D,h_D}
		\end{bNiceMatrix}.
		\label{3.8}
	\end{equation}
	Therefore, the multiplication map
$
\cdot \ell' : [B]_D \longrightarrow [B]_{D+1}
$
	has maximal rank if and only if the map
$
\cdot \ell : [A]_{D-1} \longrightarrow [A]_D
$
	is surjective.
	Thus, the proof is complete.
\end{proof}

\section{The independence polynomial of lollipop graphs}

In this section, we study the independence polynomials of lollipop graphs.  
These results will be useful to prove \Cref{SecondTheorem} in the next section.

Recall that the \emph{lollipop graph} $L_{m,n}$ is formed by attaching a complete graph $K_m$ to a path $P_n$ by a bridge, where $m,n \ge 1$ (see Figure~\ref{lollipop Lm,n}).
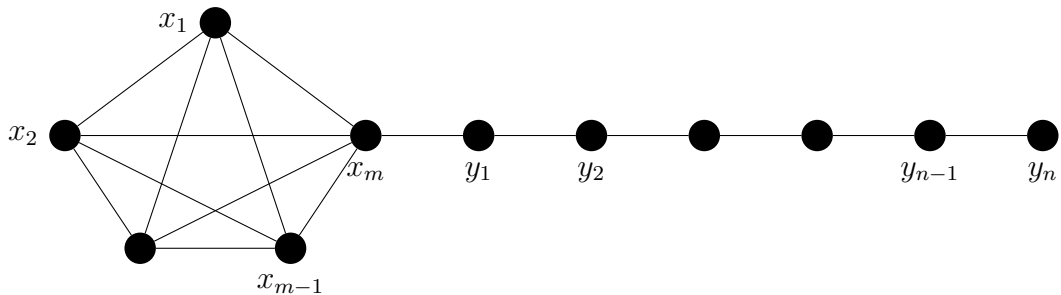
\begin{figure}[!h]
	\begin{tikzpicture}[
		every edge/.style = {draw=black,very thick},
		vrtx/.style args = {#1/#2}{
			circle, draw, thick, fill=black,
			minimum size=1mm, label=#1:#2}
		]
		\node (n1) [vrtx=left/$x_1$]at (0, 0) {};
		\node (n3) [vrtx=below/$x_m$]at (2,-1.5)  {};
		\node (n4) [vrtx=below/$x_{m-1}$]at (1,-3) {};
		\node (n5) [vrtx=below/]at (-1,-3)  {};
		\node (n6) [vrtx=left/$x_2$]at (-2,-1.5)  {};
		\node (n7) [vrtx=below/$y_1$]at (3.5,-1.5)  {};
		\node (n8) [vrtx=below/$y_2$]at (5,-1.5)  {};
		\node (n9) [vrtx=below/]at (6.5,-1.5)  {};
		\node (n10) [vrtx=below/]at (8,-1.5)  {};
		\node (n11) [vrtx=below/$y_{n-1}$]at (9.5,-1.5)  {};
		\node (n12) [vrtx=below/$y_n$]at (11,-1.5)  {};
		\foreach \from/\to in {n1/n3,n1/n4,n1/n5,n1/n6,n3/n4,n3/n5,n3/n6,n4/n5,n4/n6, n5/n6, n3/n7, n8/n7, n8/n9, n9/n10, n10/n11, n11/n12}		
		\draw (\from) -- (\to);	
	\end{tikzpicture}
	\caption{Lollipop graph $L_{m,n}$}
	\label{lollipop Lm,n}
\end{figure}
As lollipop graphs are claw-free, their independence polynomials are log-concave, and consequently unimodal~\cite{Hamidoune90}.  
Denote by $\eta_{m,n}$ the mode of the independence polynomial $I(L_{m,n};t)$, whose value we aim to estimate.
For a polynomial $f(x) = \sum_{i=0}^n a_i x^i$, we set $a_k = 0$ for all $k > n$ or $k < 0$.

\begin{lem}\label{unimodal_mode_unit}
Let $f$ and $g$ be unimodal polynomials with nonnegative real coefficients, whose modes are $\mu$ and $\nu$, respectively. 
If $\abs{\mu - \nu} \le 1$, then $f + g$ is also unimodal, and its mode belongs to $\{\min(\mu, \nu),\, \min(\mu, \nu) + 1\}$.
\end{lem}

\begin{proof}
Without loss of generality, assume that $\mu \le \nu$. 
	Write $f = \sum_{k=0}^n a_k x^k$ and $g = \sum_{k=0}^m b_k x^k$. 
	If $\mu = \nu$, the claim is immediate. 
	
	Suppose now that $\nu = \mu + 1$. 
	Then it is straightforward to verify that
	\begin{align*}
		&a_{i-1} + b_{i-1} \le a_i + b_i, \quad &&\text{for all } i = 1, \ldots, \mu - 1,\\
		&a_i + b_i \ge a_{i+1} + b_{i+1}, \quad &&\text{for all } i = \mu + 1, \ldots, \max(m,n),\\
		&a_{\mu - 1} + b_{\mu - 1} < a_\mu + b_\mu.
	\end{align*}
	If $a_\mu + b_\mu \ge a_{\mu + 1} + b_{\mu + 1}$, then $f + g$ is unimodal with mode $\mu$; 
	otherwise, if $a_\mu + b_\mu < a_{\mu + 1} + b_{\mu + 1}$, its mode is $\mu + 1$.
\end{proof}

\begin{prop}\label{mode of lollipop}
For all $m, n \ge 1$, the mode $\eta_{m,n}$ of the independence polynomial of the lollipop graph $L_{m,n}$ belongs to $\{\lambda_n, \lambda_n + 1\}$.
\end{prop}

\begin{proof}
	We fix $n$ and proceed by induction on $m$. 
	For $m = 1$, the lollipop graph $L_{1,n}$ is the path graph $P_{n+1}$. 
	Hence, the mode of $I(L_{1,n}; t)$ is $\lambda_{n+1}$. 
	By Lemma~\ref{lambda_n < lambda_n+1}, we have $\lambda_{n+1} \in \{\lambda_n, \lambda_n + 1\}$, and thus the base case holds.
	
	Assume the statement holds for some $m$, and consider $m + 1$. 
	By Proposition~\ref{FormulaforIndependence}, removing a vertex from the complete graph that is not joined to the path gives
	\[
	I(L_{m+1,n}; t) = I(L_{m,n}; t) + t\,I(P_n; t).
	\]
	By the inductive hypothesis, the mode of $I(L_{m,n}; t)$ belongs to $\{\lambda_n, \lambda_n + 1\}$, 
	while the mode of $t I(P_n; t)$ equals $\lambda_n + 1$. 
Hence, by Lemma~\ref{unimodal_mode_unit}, the mode of $I(L_{m+1,n}; t)$ belongs to $\{\lambda_n, \lambda_n + 1\}$, as desired.
\end{proof}

\begin{rem}\label{constant mode}
	From the proof of Proposition~\ref{mode of lollipop}, one can also deduce that 
	if there exists an integer $m_0$ such that the mode of $I(L_{m_0,n}; t)$ equals $\lambda_n + 1$, 
	then $\eta_{m,n} = \lambda_n + 1$ for all $m \ge m_0$. 
	Indeed, we have
	\[
	I(L_{m+1,n}; t) = I(L_{m,n}; t) + t\,I(P_n; t).
	\]
	Since the mode of $t I(P_n; t)$ equals $\lambda_n + 1$, and hence if $\eta_{m,n} = \lambda_n + 1$, then $\eta_{m+1,n} = \lambda_n + 1$ as well. 
	The conclusion follows by induction.
\end{rem}

\section{WLP of Artinian algebras associated to lollipop graphs}
We now study the weak Lefschetz property of the Artinian algebra associated to the lollipop graphs $L_{m,n}$, labeled as in Figure~\ref{lollipop Lm,n}.  
Let $R = \Bbbk[x_1, \ldots, x_m, y_1, \ldots, y_n]$ be a standard graded polynomial ring, and define
\[
A(L_{m,n}) = 
\frac{R}{(x_1^2, \ldots, x_m^2, y_1^2, \ldots, y_n^2) + I(L_{m,n})},
\]
where $I(L_{m,n})$ denotes the edge ideal of the lollipop graph $L_{m,n}$.  
Throughout this section, we assume that $\operatorname{char}(\Bbbk) = 0$, and denote by $\ell$ the sum of all variables in the polynomial ring.  
Our main result in this paper is stated as follows.

\begin{thm}\label{WLP for lollipop}
	The algebra $A(L_{m,n})$ has the WLP if and only if one of the following conditions holds:
	\begin{enumerate}[\quad \rm (i)]
		\item $m = 1$ and $n \in \{1,2,\ldots,6,8,9,12\}$;
		\item $m = 2$ and $n \in \{1,2,\ldots,5,7,8,11\}$;
		\item $m \ge 3$ and $n \in \{1,3,4,7\}$.
	\end{enumerate}
\end{thm}

To prove this theorem, we need several auxiliary results.  
By Proposition~\ref{mode of lollipop}, the mode $\eta_{m,n}$ of the independence polynomial $I(L_{m,n};t)$ satisfies $\eta_{m,n} \in \{\lambda_n, \lambda_n + 1\}$.  
We first show that $A(L_{m,n})$ fails to have the WLP whenever $\eta_{m,n} = \lambda_n$.

\begin{lem}\label{eta_m,n = lambda_n}
	Let $m \ge 3$ and $n \ge 1$ be integers such that $\eta_{m,n} = \lambda_n$. 
	Then the multiplication map
	\[
	\cdot \ell : [A(L_{m,n})]_{\eta_{m,n}} \longrightarrow [A(L_{m,n})]_{\eta_{m,n}+1}
	\]
	is not surjective.
\end{lem}

\begin{proof}
	Write $A(L_{m,n}) = R/I$, where 
	\[
	R = \Bbbk[x_1, \ldots, x_m, y_1, \ldots, y_n]
	\quad\text{and}\quad
	I = (x_1^2, \ldots, x_m^2, y_1^2, \ldots, y_n^2) + I(L_{m,n}).
	\]
	Consider the short exact sequence
	\[
	0 \longrightarrow R/(I:x_m)(-1)
	\xrightarrow{\cdot x_m}
	R/I
	\longrightarrow R/(I + (x_m))
	\longrightarrow 0,
	\]
	in which
	\[
	R/(I + (x_m))
	\cong
	\frac{\Bbbk[x_1, \ldots, x_{m-1}]}{(x_1, \ldots, x_{m-1})^2}
	\otimes_\Bbbk
	A(P_n).
	\]
	This yields the following commutative diagram with exact rows
	\[
	\xymatrix{
		[A(L_{m,n})]_{\lambda_n} \ar[r] \ar[d]^{\cdot \ell} &
		\left[
		\frac{\Bbbk[x_1, \ldots, x_{m-1}]}{(x_1, \ldots, x_{m-1})^2}
		\otimes_\Bbbk A(P_n)
		\right]_{\lambda_n} \ar[r] \ar[d]^{\cdot \ell} &
		0 \\
		[A(L_{m,n})]_{\lambda_n + 1} \ar[r] &
		\left[
		\frac{\Bbbk[x_1, \ldots, x_{m-1}]}{(x_1, \ldots, x_{m-1})^2}
		\otimes_\Bbbk A(P_n)
		\right]_{\lambda_n + 1} \ar[r] &
		0.}
	\]
Since $\lambda_n$ is the mode of $I(P_n; t)$, the multiplication map 
	\[
	\cdot \ell : [A(P_n)]_{\lambda_n - 1} \longrightarrow [A(P_n)]_{\lambda_n}
	\]
	is not surjective.  
	By Theorem~\ref{generalization of lollipop}, the map
	\[
	\cdot \ell :
	\left[
	\frac{\Bbbk[x_1, \ldots, x_{m-1}]}{(x_1, \ldots, x_{m-1})^2}
	\otimes_\Bbbk A(P_n)
	\right]_{\lambda_n}
	\longrightarrow
	\left[
	\frac{\Bbbk[x_1, \ldots, x_{m-1}]}{(x_1, \ldots, x_{m-1})^2}
	\otimes_\Bbbk A(P_n)
	\right]_{\lambda_n + 1}
	\]
	is also not surjective.  
	Hence, the map 
	\[
	\cdot \ell : [A(L_{m,n})]_{\lambda_n} \longrightarrow [A(L_{m,n})]_{\lambda_n + 1}
	\]
	fails surjectivity, as desired.
\end{proof}

\begin{ex}
Consider the lollipop graph $L_{4,9}$.  
By using {\tt Macaulay2} {\rm \cite{Macaulay2}} to compute the Hilbert series
	\[
	HS(A(L_{4,9}),t)
	= 1 + 13t + 63t^2 + 140t^3 + 140t^4 + 51t^5 + 3t^6.
	\]
It follows that $\eta_{4,9} = \lambda_9 = 3$.  
	By Lemma~\ref{eta_m,n = lambda_n}, the multiplication map $\cdot \ell$ from degree~$3$ to degree~$4$ is not surjective.  
	Hence, $A(L_{4,9})$ fails the WLP.
\end{ex}

\begin{prop}\label{8,11,14,15}
	If $n \in \{8, 11, 14, 15\}$ or $n \ge 17$, and $m \ge 3$, then the algebra $A(L_{m,n})$ fails surjectivity from degree $\lambda_n + 1$ to degree $\lambda_n + 2$.  
	Consequently, $A(L_{m,n})$ does not have the WLP.
\end{prop}

\begin{proof}
	As in the proof of Lemma~\ref{eta_m,n = lambda_n}, we have the following diagram with exact rows
	\[
	\xymatrix{
		[A(L_{m,n})]_{\lambda_n + 1} \ar[r] \ar[d]^{\cdot \ell} &
		\left[
		\frac{\Bbbk[x_1, \ldots, x_{m-1}]}{(x_1, \ldots, x_{m-1})^2}
		\otimes_\Bbbk A(P_n)
		\right]_{\lambda_n + 1} \ar[r] \ar[d]^{\cdot \ell} &
		0 \\
		[A(L_{m,n})]_{\lambda_n + 2} \ar[r] &
		\left[
		\frac{\Bbbk[x_1, \ldots, x_{m-1}]}{(x_1, \ldots, x_{m-1})^2}
		\otimes_\Bbbk A(P_n)
		\right]_{\lambda_n + 2} \ar[r] &
		0.}
	\]
	By \Cref{WLP for P_n}, the multiplication map
	\[
	\cdot \ell : [A(P_n)]_{\lambda_n} \longrightarrow [A(P_n)]_{\lambda_n+1}
	\]
	is not surjective.  
	Applying \Cref{generalization of lollipop}, we deduce that
	\[
	\cdot \ell :
	\left[
	\frac{\Bbbk[x_1, \ldots, x_{m-1}]}{(x_1, \ldots, x_{m-1})^2}
	\otimes_\Bbbk A(P_n)
	\right]_{\lambda_n+1}
	\longrightarrow
	\left[
	\frac{\Bbbk[x_1, \ldots, x_{m-1}]}{(x_1, \ldots, x_{m-1})^2}
	\otimes_\Bbbk A(P_n)
	\right]_{\lambda_n+2}
	\]
	is also not surjective.  
	Thus, the map
	\[
	\cdot \ell : [A(L_{m,n})]_{\lambda_n+1} \longrightarrow [A(L_{m,n})]_{\lambda_n+2}
	\]
	fails to be surjective.  
By \Cref{mode of lollipop}, $\eta_{m,n} \le \lambda_n + 1$, the assertion follows.
\end{proof}
\begin{prop}\label{2,5,6,9,...}
	If $n \in \{2, 5, 6, 9, 10, 12, 13, 16\}$ and $m \ge 3$, then the algebra $A(L_{m,n})$ does not have the WLP.
\end{prop}

\begin{proof}
	If $\eta_{m,n} = \lambda_n$, then the result follows from Lemma~\ref{eta_m,n = lambda_n}.  
	Hence, we may assume that $\eta_{m,n} = \lambda_n + 1$.  
Set
	\[
R/J =\frac{\Bbbk[x_1, \ldots, x_m]}{(x_1, \ldots, x_m)^2}\otimes_\Bbbk A(P_n),
	\quad
	\text{where} \quad
	A(P_n) =
	\frac{\Bbbk[y_1, \ldots, y_n]}{(y_1^2, \ldots, y_n^2) + I(P_n)}.
	\]
	Consider the short exact sequence
	\[
	0 \longrightarrow R/(J:x_m y_1)(-2)
	\xrightarrow{\cdot x_m y_1}
	R/J
	\longrightarrow R/(J + (x_m y_1))
	\longrightarrow 0,
	\]
	where
	$R/(J:x_m y_1) \cong A(P_{n-2})$
	and
	$R/(J + (x_m y_1)) \cong A(L_{m,n})$.
	This induces the following commutative diagram with exact rows
	\[
	\xymatrix@C=3em@R=3em{
		0 \ar[r] &
		[A(P_{n-2})]_{\lambda_n - 2} \ar[r] \ar[d]^{\cdot \ell} &
		\left[
		\frac{\Bbbk[x_1, \ldots, x_m]}{(x_1, \ldots, x_m)^2}
		\otimes_\Bbbk A(P_n)
		\right]_{\lambda_n} \ar[r] \ar[d]^{\cdot \ell} &
		[A(L_{m,n})]_{\lambda_n} \ar[r] \ar[d]^{\cdot \ell} &
		0 \\
		0 \ar[r] &
		[A(P_{n-2})]_{\lambda_n - 1} \ar[r] &
		\left[
		\frac{\Bbbk[x_1, \ldots, x_m]}{(x_1, \ldots, x_m)^2}
		\otimes_\Bbbk A(P_n)
		\right]_{\lambda_n + 1} \ar[r] &
		[A(L_{m,n})]_{\lambda_n + 1} \ar[r] &
		0
	}
	\]
By Theorem~\ref{WLP for P_n} and Table~\ref{tab_indpoly}, the multiplication map
\[
\cdot \ell : [A(P_{n-2})]_{\lambda_n - 2} \to [A(P_{n-2})]_{\lambda_n - 1}
\]
is injective, for all $n \in \{2,5,6,9,10,12,13,16\}$.  
By using {\tt Macaulay2} \cite{M2codes} to compute the Hilbert series of  $A(P_n)$, we see that
	$\dim_\Bbbk [A(P_n)]_{\lambda_n - 1}
	\ge \dim_\Bbbk [A(P_n)]_{\lambda_n + 1}$,
	so the map
	$\cdot \ell^2 : [A(P_n)]_{\lambda_n - 1}
	\to [A(P_n)]_{\lambda_n + 1}$
	is not injective.  
	By Theorem~\ref{generalization of lollipop}, the map
	\[
	\cdot \ell :
	\left[
	\frac{\Bbbk[x_1, \ldots, x_m]}{(x_1, \ldots, x_m)^2}
	\otimes_\Bbbk A(P_n)
	\right]_{\lambda_n}
	\longrightarrow
	\left[
	\frac{\Bbbk[x_1, \ldots, x_m]}{(x_1, \ldots, x_m)^2}
	\otimes_\Bbbk A(P_n)
	\right]_{\lambda_n + 1}
	\]
	is also not injective.  
Hence, the right vertical map of the diagram 
\[
\cdot \ell : [A(L_{m,n})]_{\lambda_n}\to [A(L_{m,n})]_{\lambda_n + 1}
\]
fails injectivity.  Since $\eta_{m,n} = \lambda_n + 1$, so $A(L_{m,n})$ does not have the WLP. This concludes the proof.
\end{proof}

We now turn to the remaining cases $n \in \{1,3,4,7\}$.  Unlike the previous cases, $A(L_{m,n})$ has the WLP for all $m \ge 3$.  
To verify this, we first compute the Hilbert series of $A(L_{3,n})$ for $n \in \{1,3,4,7\}$ by using {\tt Macaulay2}~\cite{Macaulay2}
\[
\begin{aligned}
	&HS(A(L_{3,1}),t) = 1 + 4t + 2t^2,\\
	&HS(A(L_{3,3}),t) = 1 + 6t + 9t^2 + 2t^3,\\
	&HS(A(L_{3,4}),t) = 1 + 7t + 14t^2 + 7t^3,\\
	&HS(A(L_{3,7}),t) = 1 + 10t + 35t^2 + 50t^3 + 25t^4 + 2t^5.
\end{aligned}
\]
By Table~\ref{tab_indpoly}, $\lambda_1 = 0$, $\lambda_3 = \lambda_4 = 1$, and $\lambda_7 = 2$. Therefore,  by Remark~\ref{constant mode}, the mode of $I(L_{m,n};t)$ is $\eta_{m,n}=\lambda_n + 1$ for all $m \ge 3$ and $n \in \{1,3,4,7\}$.  
We are thus ready to establish the following proposition.

\begin{prop}\label{1,3,4,7}
	For all $m \ge 3$ and $n \in \{1,3,4,7\}$, the algebra $A(L_{m,n})$ has the WLP.
\end{prop}
\begin{proof}
When $m = 3$, the lollipop graph $L_{3,n}$ coincides with the tadpole graph $T_{3,n}$.  Hence, by {\rm \cite[Corollary~4.5]{NT2024}}, the algebra $A(L_{3,n})$ has the WLP for all $n \in \{1,3,4,7\}$.
	
Fix $n \in \{1,3,4,7\}$. We prove the assertion by induction on $m$.  The base case $m = 3$ has been established above.  
Assume that the statement holds for some $m \ge 3$, and we prove it for $m + 1$.
	
Write $A(L_{m+1,n}) =R/I$ and consider the short exact sequence
	\[
	0 \longrightarrow R/(I:x_1)(-1)
	\xrightarrow{\cdot x_1}
	R/I
	\longrightarrow R/(I + (x_1))
	\longrightarrow 0,
	\]
	where $R/(I:x_1) \cong A(P_n)$ and $R/(I + (x_1)) \cong A(L_{m,n})$.  
	This yields the commutative diagram with exact rows
	\[
	\xymatrix{
		0 \ar[r] &
		[A(P_n)]_{i-1} \ar[r] \ar[d]^{\cdot \ell} &
		[A(L_{m+1,n})]_i \ar[r] \ar[d]^{\cdot \ell} &
		[A(L_{m,n})]_i \ar[r] \ar[d]^{\cdot \ell} & 0 \\
		0 \ar[r] &
		[A(P_n)]_i \ar[r] &
		[A(L_{m+1,n})]_{i+1} \ar[r] &
		[A(L_{m,n})]_{i+1} \ar[r] & 0.
	}
	\]
For all $i \ge \eta_{m,n}=\lambda_n + 1$, the multiplication map
	$\cdot \ell : [A(P_n)]_{i-1} \to [A(P_n)]_i$
	is surjective since $A(P_n)$ has the WLP by Theorem~\ref{WLP for P_n}.  
	By the induction hypothesis,
	$\cdot \ell : [A(L_{m,n})]_i \to [A(L_{m,n})]_{i+1}$
	is also surjective.  
	Hence, applying the snake lemma, we deduce that
	$\cdot \ell : [A(L_{m+1,n})]_i \to [A(L_{m+1,n})]_{i+1}$
	is surjective.  
	A similar argument shows that $\cdot \ell : [A(L_{m+1,n})]_i \to [A(L_{m+1,n})]_{i+1}$ is injective for all $i \le \eta_{m,n}-1=\lambda_n$.  
Therefore, $A(L_{m+1,n})$ has the WLP.  
Hence, by induction, the claim holds for all $m \ge 3$.
\end{proof}

We are now ready to prove \Cref{WLP for lollipop}.

\begin{proof}
	When $m = 1$ (respectively, $m = 2$), the lollipop graph $L_{m,n}$ is the path graph $P_{n+1}$ (respectively, $P_{n+2}$).  
	Hence, by Theorem~\ref{WLP for P_n}, the theorem holds in these two cases.  
	
	It remains to consider the case $m \ge 3$.  
	We divide the proof into three cases according to the value of $n$.
	
	\noindent\textbf{Case 1:} $n \in \{8,11,14,15\}$ or $n \ge 17$.  
	In this case, $A(L_{m,n})$ fails to have the WLP by Proposition~\ref{8,11,14,15}.
	
	\noindent\textbf{Case 2:} $n \in \{2,5,6,9,10,12,13,16\}$.  
	In this case, $A(L_{m,n})$ also fails to have the WLP by Proposition~\ref{2,5,6,9,...}. 
	
	\noindent\textbf{Case 3:} $n \in \{1,3,4,7\}$.  
	Here, $A(L_{m,n})$ has the WLP by Proposition~\ref{1,3,4,7}.
	
Therefore, the proof is complete.
\end{proof}

\section*{Acknowledgments}
This research was supported by Hue University of Education's Project under grant number NCTBSV.T.25--TN--101--01.   
The first author also acknowledges the partial support of Hue University under the Core Research Program, Grant No.~NCTB.DHH.2024.01.


\bibliographystyle{plain} 
\bibliography{WLP_Graphs} 

\begin{thebibliography}{10}

\bibitem{AB2020}
N.~Altafi and M.~Boij.
\newblock The weak {L}efschetz property of equigenerated monomial ideals.
\newblock {\em J. Algebra}, 556:136--168, 2020.

\bibitem{AN2018}
N.~Altafi and N.~Nemati.
\newblock Lefschetz properties of monomial algebras with almost linear
  resolution.
\newblock {\em Comm. Algebra}, 48(4):1499--1509, 2020.

\bibitem{BMMNZ12}
M.~Boij, J.C. Migliore, R.M. Mir\'{o}-Roig, U.~Nagel, and F.~Zanello.
\newblock On the shape of a pure {$O$}-sequence.
\newblock {\em Mem. Amer. Math. Soc.}, 218(1024):viii+78, 2012.

\bibitem{Cooperetal2023}
S.M. Cooper, S.~Faridi, T.~Holleben, L.~Nicklasson, and A.~Van~Tuyl.
\newblock The weak {L}efschetz property of whiskered graphs.
\newblock In Uwe Nagel, Karim Adiprasito, Roberta Di~Gennaro, Sara Faridi, and
  Satoshi Murai, editors, {\em Lefschetz Properties}, pages 97--110, Singapore,
  2024. Springer Nature Singapore.

\bibitem{DaoNair2022}
H.~Dao and R.~Nair.
\newblock On the {L}efschetz property for quotients by monomial ideals
  containing squares of variables.
\newblock {\em Communications in Algebra}, 52(3):1260--1270, 2024.

\bibitem{GLN2020}
O.~Gasanova, S.~Lundqvist, and L.~Nicklasson.
\newblock On decomposing monomial algebras with the {L}efschetz properties.
\newblock {\em J. Pure Appl. Algebra}, 226(6):Paper No. 106968, 15, 2022.

\bibitem{Macaulay2}
D.R. Grayson and M.E. Stillman.
\newblock Macaulay2, a software system for research in algebraic geometry.
\newblock Available at \url{http://www2.macaulay2.com}.

\bibitem{GH83}
I.~Gutman and F.~Harary.
\newblock Generalizations of the matching polynomial.
\newblock {\em Utilitas Math.}, 24:97--106, 1983.

\bibitem{Hamidoune90}
Y.O. Hamidoune.
\newblock On the numbers of independent {$k$}-sets in a claw free graph.
\newblock {\em J. Combin. Theory Ser. B}, 50(2):241--244, 1990.

\bibitem{HMMNWW2013}
T.~Harima, T.~Maeno, H.~Morita, Y.~Numata, A.~Wachi, and J.~Watanabe.
\newblock {\em The {L}efschetz properties}, volume 2080 of {\em Lecture Notes
  in Mathematics}.
\newblock Springer, Heidelberg, 2013.

\bibitem{HL94}
C.~Hoede and X.L. Li.
\newblock Clique polynomials and independent set polynomials of graphs.
\newblock {\em Discrete Math.}, 125:219--228, 1994.

\bibitem{Holleben2024}
T.~Holleben.
\newblock The weak {L}efschetz property and mixed multiplicities of monomial
  ideals.
\newblock {\em J. Algebraic Combin.}, 60(1):295--317, 2024.

\bibitem{Holleben2025}
T.~Holleben.
\newblock Lefschetz properties of squarefree monomial ideals via rees algebras.
\newblock {\em Journal of Algebra}, 668:572--599, 2025.

\bibitem{Kling2024}
F.~Jonsson~Kling.
\newblock The strong {L}efschetz property for quadratic reverse lexicographic
  ideals.
\newblock {\em Proc. Amer. Math. Soc. Ser. B}, 11:390--401, 2024.

\bibitem{K2025}
F.~Jonsson~Kling.
\newblock Preserving lefschetz properties after extension of variables.
\newblock {\em arXiv preprint arXiv:2503.16990}, 2025.

\bibitem{MMO2013}
E.~Mezzetti, R.M. Mir\'{o}-Roig, and G.~Ottaviani.
\newblock Laplace equations and the weak {L}efschetz property.
\newblock {\em Canad. J. Math.}, 65(3):634--654, 2013.

\bibitem{MM2016}
M.~Micha\l{ek} and R.M. Mir\'{o}-Roig.
\newblock Smooth monomial {T}ogliatti systems of cubics.
\newblock {\em J. Combin. Theory Ser. A}, 143:66--87, 2016.

\bibitem{MNS2020}
J.~Migliore, U.~Nagel, and H.~Schenck.
\newblock The weak {L}efschetz property for quotients by quadratic monomials.
\newblock {\em Math. Scand.}, 126(1):41--60, 2020.

\bibitem{MMN2011}
J.C. Migliore, R.M. Mir\'{o}-Roig, and U.~Nagel.
\newblock Monomial ideals, almost complete intersections and the weak
  {L}efschetz property.
\newblock {\em Trans. Amer. Math. Soc.}, 363(1):229--257, 2011.

\bibitem{MN2013}
J.C. Migliore and U.~Nagel.
\newblock Survey article: a tour of the weak and strong {L}efschetz properties.
\newblock {\em J. Commut. Algebra}, 5(3):329--358, 2013.

\bibitem{M2codes}
H.D. Nguyen and Q.H. Tran.
\newblock Macaulay2 codes for checking the weak {L}efschetz property of
  artinian algebra associated to graphs.
\newblock Available at
  \url{https://sites.google.com/view/tranquanghoassite/software?authuser=0},
  2024.

\bibitem{NT2024}
H.D. Nguyen and Q.H. Tran.
\newblock The weak {L}efschetz property of artinian algebras associated to
  paths and cycles.
\newblock {\em Acta Math Vietnam}, 49(3):523–544, 2024.

\bibitem{HT2023}
H.V.N. Phuong and Q.H. Tran.
\newblock A new proof of {S}tanley's theorem on the strong {L}efschetz
  property.
\newblock {\em Colloq. Math.}, 173(1):1--8, 2023.

\bibitem{QHT2020}
Q.H. Tran.
\newblock The {L}efschetz properties of artinian monomial algebras associated
  to some graphs.
\newblock {\em Journal of Science, Hue University of Education}, 59(3):12--22,
  2021.

\end{thebibliography}

\end{document}